\documentclass[a4paper,11pt]{amsart}
\usepackage{amssymb,amsfonts,amsxtra,
mathrsfs,placeins,graphicx,verbatim,stmaryrd, comment}
\usepackage[all]{xy}
\xyoption{line}
\usepackage{fullpage}
\usepackage{euscript}
\usepackage[textsize=footnotesize]{todonotes}

\newtheorem{theorem}{Theorem}[section]
\newtheorem{cor}[theorem]{Corollary}
\newtheorem{lem}[theorem]{Lemma}
\newtheorem{prop}[theorem]{Proposition}

\theoremstyle{definition}

\newtheorem{example}[theorem]{Example}
\newtheorem{defi}[theorem]{Definition}
\newtheorem{rem}[theorem]{Remark}

\numberwithin{equation}{section}
\DeclareMathOperator{\Hom}{Hom}

\DeclareMathOperator{\abs}{abs}

\def\ground{\mathbf k}

\def\End{\operatorname{End}}

\DeclareMathOperator{\Hoch}{Hoch}
\DeclareMathOperator{\Hochb}{\overline{Hoch}}
\def\modcat{\operatorname{\!-mod}}
\def\op{\operatorname{op}}
\def\ps{\operatorname{ps}}

\begin{document}

\title[Koszul-Morita duality]{Koszul-Morita duality}
\author{$\rm{JOSEPH\,\, CHUANG}^1$\,,\,\,  $\rm{ANDREY\,\, LAZAREV}^2$ \,\,\& $\rm{W.H.\,\, MANNAN}^3$}
\maketitle

\footnotetext[1]{Department of Mathematics, City University London, London EC\textup{1V 0}HB. email: j.chuang@city.ac.uk}
\footnotetext[2]{Mathematics and Statistics, Lancaster University, Lancaster LA\textup{1 4}YF. email: a.lazarev@lancaster.ac.uk}
\footnotetext[3]{Mathematics and Statistics, Lancaster University, Lancaster LA\textup{1 4}YF. email: wajid@mannan.info}

\begin{abstract}
We construct a generalization of  Koszul duality in the sense of Keller--Lef\`evre for not necessarily augmented algebras. This duality is closely related to classical Morita duality and specializes to it in certain cases.
\end{abstract}

\section{Introduction}
Koszul duality is an anti-equivalence between certain subcategories of the derived category of a quadratic Koszul algebra $A$ and that of its Koszul dual $A^!$ \cite{BGS}. More recently Keller and Lef\`evre  \cite{K-L} gave a general formulation which is valid for a general augmented differential graded (dg) algebra  and its dg Koszul dual (which is in that case naturally a dg \emph{coalgebra}). The derived category on the dual side is that of dg \emph{comodules} over this dg coalgebra. In fact, in the coalgebra-comodule language the Koszul equivalence becomes covariant, but taking linear duals restores contravariance. We adopt the set-up in which the Koszul correspondence is contravariant; this allows us to replace coalgebras and comodules with the dual notions of pseudo-compact algebras and pseudo-compact modules, cf. \cite{VdB} concerning these notions. This language is of course equivalent to the language of coalgebras and comodules.

It is natural to ask whether the assumption that $A$ be augmented is essential for constructing Koszul duality. One answer which completely removes this assumption  is provided by the work of Positselski \cite{P}, but at a price: the Koszul dual object op. cit. is no longer a pseudo-compact dg algebra but a more general one, called a (pseudo-compact) \emph{curved} dg algebra.  The category of modules also needs to be appropriately modified.

We provide a different answer, which ensures that (the analogue of) the Koszul dual object is still a pseudo-compact dg algebra, even though the original $A$ is not necessarily augmented. It is however, required to possess  a non-zero finite-dimensional dg module $M$ (in the case $M= \ground$, the ground field, this reduces to the Keller-Lef\`{e}vre treatment). We prove further, that such a dg module always exists, at least if we replace $A$ by a quasi-isomorphic dg algebra. The unexpected conclusion is that the derived category of dg modules over a dg algebra is always equivalent to some derived category of pseudo-compact modules.

Conversely, we give a necessary and sufficient condition for the derived category of pseudo-compact dg modules over a pseudo-compact dg algebra to be anti-equivalent to the derived category of dg modules over some dg algebra. In contrast, this condition does not always hold, i.e.\ there are derived categories of pseudo-compact modules which are not equivalent to any derived module category.

Furthermore, the functor associating to a dg $A$--module an appropriate pseudo-compact module on the Koszul dual side is given as a kind of (derived) Hom into $M$. It is, thus, reminiscent of the (derived version of) the classical Morita duality \cite{Morita, Lam}. It turns out that our duality could indeed be viewed as an extension of Morita duality in the case when $A$ is an ordinary finite-dimensional algebra of \emph{finite global dimension}.

It seems likely that a large portion of our results could be extended to dg modules over dg categories, however we refrained from working in this generality to keep exposition simple. Related results, in the context of dg categories, are contained in the recent preprint \cite{Prid}. The main difference in our approach is that the use of the reduced Hochschild complex allowed us to avoid additional assumptions present in e.g.\ Proposition 3.9 of op. cit.

It also is worth noting that equivalences between categories of pseudo-compact modules (phrased in the language of comodules) were studied in the work of Takeuchi \cite{Tak} and so our results could be viewed as linking (derived versions of) Morita theory and Takeuchi theory.

\subsection{Organization of the paper}In sections 2 through 7 we formulate and prove our main result: a Quillen anti-equivalence between categories of dg modules over a dg algebra $A$, possessing a finite-dimensional non-zero dg module $M$, and the category of pseudo-compact dg modules over a pseudo-compact dg algebra $E$, which is the reduced Hochschild complex of $A$ with coefficients in $\End M$. We note that $E$ computes the derived endomorphisms of $M$ as an $A$--module, just as (for example) the corresponding \emph{unreduced} complex. However we cannot replace $E$ with the corresponding unreduced complex or use some other resolution of $M$ since that will change the derived category of pseudo-compact modules. We do not have a satisfactory explanation of this striking phenomenon; perhaps it is related to the lack of an appropriate closed model category structure on all pseudo-compact dg algebras (or, equivalently, all dg coalgebras).

Our main tool in proving the mentioned  Quillen equivalence is Koszul duality as developed by Positselski \cite{P}. One minor modification that we introduce is the systematic use of (pseudo-compact) dg modules and their Maurer-Cartan twisting which might be of some independent interest. It is worth remarking that the end result does not  involve curved dg algebras and modules even though Positselski's results and our use of them do.

In Section 8 we consider the question of when the category of pseudo-compact dg modules is Quillen equivalent to some (ordinary) derived category of dg modules; a particularly simple criterion is formulated in the case of an ungraded pseudo-compact algebra. Finally Section 9 explains how our results  essentially reduce to (derived) Morita duality in the case when $A$ is  finite-dimensional and of finite global dimension.

\section{Formulation of the main result}
Let $A$ be a dg algebra over a field $\ground$ and let $M$ be a non-zero finite-dimensional  dg module over $A$.  All dg algebras and modules will be cohomologically graded (so the differential will have degree 1).  Recall (cf., for example \cite{P}) that the category of dg $A$--modules has the structure of a closed model category where weak equivalences are quasi-isomorphisms and fibrations are the surjective maps. All $A$--modules are fibrant and cofibrant objects are retracts of cell $A$--modules; the latter are the $A$--modules having a filtration whose associated factors are free $A$--modules. We can form the reduced Hochschild complex of $A$ with coefficients in $\End(M)$:
\[
\Hochb^n(A,\End M)\subset\Hom(A^{\otimes n}, \End M)
\]
consisting of the reduced cochains; that is multilinear maps (over $\ground$) $A^{\times n}\to \End M$ which vanish if any of the arguments is $1\in A$. The complex $\Hochb^\bullet(A,\End M)$ is a dg algebra which we will denote by $E$.

To see that it is a dg algebra first define multiplication of a degree $n$ element $\alpha$ with a degree $m$ element $\beta$ by setting application to elements $a_1,\cdots,a_n,b_1,\cdots,b_m\in A$ to be given by: \begin{eqnarray} \label{mulinE}
(\alpha\beta) (a_1,\cdots,a_n,b_1,\cdots,b_m)=(-1)^{st}\alpha(a_1,\cdots,a_n)\beta(b_1,\cdots,b_m),
\end{eqnarray}
where $s$ is the sum of the degrees of the $a_i$ and $t$ is the degree of $\beta$.  The multiplication on the right hand side of his equation is just composition in End$(M)$.

Note that $E$ has a differential $d$ which is the sum of:

\bigskip

\noindent 1) The internal differentials on the copies of $A$ in $A^{\otimes n}$,

\noindent 2) The differential on End$(M)$ coming from the internal differential on $M$,

\noindent 3) Contractions on the copies of $A$, including the two `end  terms' represented by taking the commutator $[\delta,\_]$ in $E$, where $\delta\colon A \to {\rm End}(M)$ is just the $A$--action on $M$.

\bigskip
\noindent Without these `end terms', the differential of a reduced cochain would no longer be reduced.

Moreover, since $E$ can be identified (disregarding the differential) with $\hat{T}\Sigma( (A/\ground)^*)\otimes \End M$ where here and elsewhere in the paper the tensor product is understood in the appropriate \emph{completed} sense. Since $M$ is finite-dimensional, the dg algebra $E$ is pseudo-compact (or, equivalently, is dual to a dg coalgebra). Here $\Sigma$ denotes the suspension operator, which raises the degree of each homogeneous element in a graded algebra/module.
In the case $M=\ground$, the pseudo-compact dg algebra $E$ is also local (which is the same as saying that its dual dg coalgebra is conilpotent).

We consider the category of left pseudo-compact dg modules $E$-mod (or comodules in the dual setting).

It has the structure of a closed model category in which the weak equivalences
 are essentially filtered quasi-isomorphisms and fibrations are surjections \cite[Theorem 8.2]{P}. All $A$--modules are fibrant and the cofibrant $A$--modules are the retracts of free $A$--modules (disregarding the differential).

Here is our main result.
\begin{theorem}\label{mainres}
The categories $A\modcat$ and $(E\modcat)^{\op}$ are Quillen equivalent. The functor $F\colon A\modcat\to (E\modcat)^{\op}$, effecting this equivalence, associates to a dg $A$--module $N$ the following $E$--module{\rm :}
\[
F(N)=\Hochb(A,\Hom(N,M)),
\]
the reduced Hochschild complex of $A$ with coefficients in $\Hom(N,M)$.
\end{theorem}
\begin{rem}
\begin{enumerate}
\item
We use (\ref{mulinE}) to define a left $E$--action on $F(N)$, with multiplication on the right hand side of (\ref{mulinE}) now coming from the  $\End M$--action on Hom$(N,M)$.
\item
The complex $\Hochb(A,\End(M))$ is quasi-isomorphic to the \emph{unreduced} Hochschild complex $\Hoch(A,\End(M))$. However these pseudo-compact dg algebras do \emph{not} have equivalent derived categories of pseudo-compact modules. For example, taking $A=\ground$ and $M=\ground$ we have $\Hoch(A,\ground)\cong\Hochb(A\times\ground,\ground)$; thus the category of pseudo-compact dg $\Hoch(A,\ground)$--modules is Quillen equivalent to the category of dg $(\ground\times\ground)$--modules. On the other hand, $\Hochb(A,\ground)\cong\ground$. Thus, the derived categories of pseudo-compact $\Hochb(A,\ground)$--modules and of pseudo-compact $\Hoch(A,\ground)$--modules cannot be equivalent and the reduced Hochschild complex cannot be replaced with the unreduced one in the definition of the pseudo-compact algebra $E$.
\item It is possible to describe the adjoint functor $G\colon(E\modcat)^{\op}\to A\modcat$ explicitly (and this will be done later). However for this we need to develop the language of twistings and it will be in less traditional terms then the functor $F$.
\end{enumerate}
\end{rem}
If the dg $A$-module $M$ is the ground field $\ground$ (so that $A$ is augmented) then Theorem \ref{mainres} is the ordinary dg Koszul duality. However we stress that $M$ could be an arbitrary non-zero finite-dimensional dg $A$-module; in particular it could be \emph{acyclic}. This leads to the following result.
\begin{cor}\label{modcomod}
For any dg algebra $A$ the category $A$-mod is Quillen equivalent to the category $(E\modcat)^{\op}$ for some pseudo-compact dg algebra $E$.
\end{cor}
\begin{proof}
Let $C$ be any finite-dimensional acyclic dg algebra; the smallest example is a two-dimensional one having basis $\{1,x\}$ with $x^2=0$ and $d(x)=1$. Then the projection $A\times C\to A$ is clearly a quasi-isomorphism and leads to a Quillen equivalence between the categories of dg $A$-modules and $A\times C$-modules. But the dg algebra $A\times C$ has a non-zero finite dimensional module, namely $C$ via the projection $A\times C\to C$. We conclude, by Theorem \ref{mainres}, that the category $A\times C\modcat$ and thus, the category $A$-mod, is Quillen equivalent to the category $(E\modcat)^{\op}$ where $E=\overline{\Hoch}(A\times C, \End C)$.
\end{proof}
\section{Curved algebras and modules}
Recall from \cite{P} that a curved dg algebra is a graded algebra $A$ together with a `differential' $d\colon A\to A$, a derivation of $A$ of degree one such that $d^2(a)=[h,a]$ where $a\in A$ and $h\in A$ is an element of degree two such that $d(h)=0$, called the \emph{curvature} of $A$.

A morphism of curved dg algebras
$f\colon  B\to A$ is a pair $(f,a)$ consisting of a morphism of graded algebras $f\colon  B\to A$
and an element $a\in A^1$ satisfying the equations:
\begin{eqnarray*}f(d_B(x)) = d_A(f(x)) + [a, f(x)],\\
f(h_B) = h_A + d_A(a) + a^2,\end{eqnarray*}
for all $x\in B$, where $B = (B, d_B, h_B)$ and $A = (A, d_A, h_A)$.

The composition of
morphisms is defined by the rule $(f, a) \circ (g, b) = (f \circ g, a+f(b))$. Identity morphisms
are the morphisms $(id, 0)$.

A morphism $(f,0)$ as above is called strict; thus strict morphisms preserve curvature elements.

Note that isomorphisms of curved dg algebras are precisely pairs $(f,a)$ where $f$ is a linear isomorphism since the inverse map is then given by the pair $(f^{-1}, -f^{-1}(a))$. Any dg algebra can be viewed as a curved dg algebra with the zero curvature. There are non-isomorphic dg algebras which are isomorphic as curved dg algebras.

A left dg module $(M, d_M)$ over a curved dg algebra $A$ is a graded left $A$--module M  endowed with a derivation $d_M \colon  M\to M$ compatible with the derivation $d_A$ and such that
$d^2_M(x) = h_Ax$ for any $x\in M$.  Equivalently, a graded $\ground$--vector space $M$ with degree 1 derivation  $d_M \colon  M\to M$ is a left $A$--module if it is equipped with a strict map $A\to \End M$ .  Here the `differential' on $\End M$ is taking the commutator with $d_M$ and the curvature is $d_M^2$.
\begin{rem}
A curved dg algebra $A$ is \emph{not} necessarily a left (or right) module over itself. Indeed if it were, then $d_A^2=[h,-]=l_h$, where $l_h$ is the operator of left multiplication by $h$; clearly this
holds if and only if $h=0$.
 On the other hand, it is consistent to consider $A$ as an $A$-bimodule, i.e.\ a module over the curved dg algebra $A\otimes A^{\op}$ with the curvature element $h\otimes 1-1\otimes h$.
\end{rem}
We will be interested in the pseudo-compact versions of the above notions, i.e.\ pseudo-compact curved dg algebras and modules over them.
\begin{example}
Let $A$ be a \emph{unital} dg algebra.
Choose a linear map  $\epsilon\colon A\to\ground$, to be regarded as an `augmentation', even though not required to
be multiplicative nor differential.
Then the pseudo-compact graded algebra $\hat{T}\Sigma( (A/\ground)^*)$ has the structure of a curved pseudo-compact algebra. To define it view $\hat{T}\Sigma( (A/\ground)^*)$ as a `Hochschild complex' of $A$ with coefficients in $\ground$. In other words, define the `differential' on $\hat{T}\Sigma( (A/\ground)^*)$ by the formula{\rm :}
\begin{eqnarray}\label{hoc}
df(a_1,\ldots, a_n)&=& \epsilon(a_1)f(a_2,\ldots, a_n)\\\nonumber&+&\sum_{k=1}^n(-1)^{k} f(a_1,\ldots, a_ka_{k+1},\ldots, a_n)\\\nonumber&+&(-1)^{n+1}f(a_1,\ldots, a_{n-1})\epsilon (a_n).
\end{eqnarray}
 (here for simplicity it is assumed that all elements and $f$ are of even degree).

Since $\epsilon$ is not necessarily an augmentation, $d$ may not square to zero; however this will define a curved dg algebra structure on $\hat{T}\Sigma( (A/\ground)^*)$. If $\epsilon$ is chosen to be a dg map, the curvature element is the `homutator' of $\epsilon${\rm;} $h\in \hat{T}^2\Sigma( (A/\ground)^*)${\rm:} $h(a,b)=\epsilon(ab)-\epsilon(a)\epsilon(b)$. It vanishes iff $\epsilon$ is a genuine augmentation in which case $\hat{T}\Sigma( (A/\ground)^*)$ becomes uncurved and is isomorphic to $\Hochb(A,\ground)$.

If $A$ is acyclic then one cannot choose $\epsilon$ to be a dg map.  In this case the curvature element has an additional term - the differentiator $h_d\in \Sigma (A/\ground)^*)$, satisfying $h_d(a)=\epsilon(da)$.
\end{example}

The last example is the main reason (for us) to consider curved dg algebras. For any unital dg algebra $A$ it determines a (local) pseudo-compact algebra $\hat{T}\Sigma( (A/\ground)^*)$. The correspondence $A\mapsto\hat{T}\Sigma( (A/\ground)^*)$ depends on choosing a fake augmentation but any two choices are canonically isomorphic and thus, it could be viewed as a functor from dg algebras into local curved pseudo-compact dg algebras.  Positselski shows that this gives a Quillen anti-equivalence between dg algebras and local pseudo-compact curved dg algebras \cite{P}. This is a unital analogue of Keller-Lefevre's correspondence. There is also a Quillen anti-equivalence between $A$--modules and $\hat{T}\Sigma( (A/\ground)^*)$--modules which will be instrumental in establishing our main result.

\section{Twisting}
The notion of Maurer-Cartan (MC) twisting of dg algebras or dg Lie algebras is well-documented \cite{CL}. Here we will discuss twistings of \emph{curved} algebras and modules over them. Note that Positselski (and others) work with twisted cochains rather than with MC elements; our point of view is essentially equivalent but more convenient since it allows one to avoid coalgebras and comodules.
\begin{defi}\begin{enumerate}
\item
Let $A=(A,d,h)$ be a curved dg algebra and $\xi\in A^1$. The twisting of $A$ by $\xi$, denoted by $A^\xi$, is the curved dg algebra having the same underlying space as $A$, the twisted differential $d^\xi=d+[\xi,\_]$ and the twisted curvature $h^\xi=h+d\xi+\frac{1}{2}[\xi,\xi]$.
\item
Let $M,d_M$ be a module over a curved dg algebra $A$ as above. The twisting of $M$ by $\xi$, denoted by $M^{[\xi]}$, is the module over $A^\xi$ having the same underlying space as $A$ and the twisted differential
$d^{[\xi]}:=d_M+\xi$.
\end{enumerate}
\end{defi}
\begin{rem}
An uncurved dg algebra $A$ may be viewed as a module over itself, so given $\xi\in A^1$ we have that $A^{[\xi]}$ is a module over $A^\xi${\rm;} here $A^{[\xi]} \neq A^{\xi}$ as they have different differentials.
\end{rem}
\begin{defi}
Let $(A,d,h)$ be a curved dg algebra; then $\xi\in A^1$ is called MC if $h+d\xi+\frac{1}{2}[\xi,\xi]=0$.
\end{defi}
\begin{rem}
The twisting of a curved dg algebra by an MC element is an \emph{uncurved} algebra.
\end{rem}
\begin{example}\label{bigexample}\begin{enumerate}
\item
Let $A$ be a unital dg algebra, then $BA:=\hat{T}\Sigma(A^*)$ is an (uncurved) acyclic local pseudo-compact algebra. Tensoring it with $A$ we get another dg algebra (although not pseudo-compact unless $A$ is finite dimensional){\rm:} $C=BA\otimes A$. The algebra $C$ has a canonical MC element{\rm;} choosing a basis $\{e_i\}$ in $A$ and the dual basis $\{e^i\}$ in $\Sigma( A^*)$ it is{\rm:} $$\xi:=\sum e^i\otimes e_i\in BA\otimes A.$$ Then $(BA\otimes A)^\xi$ is isomorphic to $\Hoch(A,A)$, the (unreduced) Hochschild complex of $A$ with coefficients in itself. If $C$ is a (unital) dg algebra supplied with a dg algebra map $A\to C$ then $\Hoch(A,C)$ can similarly be constructed as an algebra twisting of $BA\otimes C$.

Furthermore, let $M$ be an $A$-bimodule. Then viewing $BA\otimes M$ as a bimodule over $BA\otimes A$ we can form the twisted module $(BA\otimes M)^{[\xi\otimes 1 +1\otimes \xi]}$.  Here $\xi\otimes 1+1\otimes \xi \in (BA\otimes A)\otimes (BA\otimes A)^{\op}$.  This results in what normally is denoted by $\Hoch(A,M)$, the unreduced Hochschild complex of $A$ with coefficients in $M${\rm;} it is thus naturally a bimodule over $\Hoch(A,A)$.
\item
Let $A$ be as above and consider a fake augmentation $\xi\colon A\to\ground$ as an element in $BA$. Twisting  by $\xi$ gives a local curved pseudo-compact algebra $BA^\xi$ with curvature denoted by $w$.

The reduced bar-construction $\overline{B}A:=\hat{T}\Sigma( (A/\ground)^*)$ is a (curved) subalgebra in $BA^\xi${\rm;} in fact there is an isomorphism $BA^\xi\cong \overline{B}A\langle\langle x\rangle\rangle$ mapping $\xi \mapsto x$,  where $dx=x^2+w$.

In the case when $\xi$ is a genuine augmentation we have $w=0$ and the inclusion $\overline{B}A\hookrightarrow BA$ is a quasi-isomorphism as $BA$ is the coproduct of  $\overline{B}A$ with an acyclic dg algebra.
\item Let again $A$ be a unital dg algebra and consider the curved pseudo-compact algebra
$\overline{B}A=\hat{T}\Sigma( (A/\ground)^*)$. Let $A_+$ be the kernel of the (fake) augmentation
$\epsilon\colon A\to \ground$.  Then $(A/k)^*$ can be identified with $A_+^*$ and
$\overline{B}A$ with $\hat{T}\Sigma( A_+^*)$. Thus, consider the curved dg algebra
$\hat{T}\Sigma (A_+^*)\otimes A$. It has a canonical element $\xi=\sum e^i\otimes e_i$
where $e_i$ is a basis in $A_+$ and $e^i$ is the dual basis in $\Sigma (A_+^*)$.
It turns out to be an MC element so twisting by it results in an \emph{uncurved} dg algebra $(\hat{T}\Sigma( A_+^*)\otimes A)^\xi$. The latter dg algebra is $\Hochb(A,A)$, the reduced Hochschild complex of $A$ with coefficients in itself. If $C$ is a (unital) dg algebra supplied with a dg algebra map $A\to C$ then $\Hochb(A,C)$ can similarly be constructed as an algebra twisting of $\overline{B}A\otimes C$.

If $M$ is an $A$-bimodule one can similarly form $\Hochb(A,M)$ as a bimodule over $\Hochb(A,A)$.
\item
Let $M$ be a \emph{left} dg module over a dg algebra $A$ and consider $\Hoch(A,A)=(BA\otimes A)^\xi$ as above and its left module $(BA\otimes M)^{[\xi]}$.  This module looks similar to $\Hoch(A,M)$ but the differential is slightly different: one has to omit the last term in (\ref{hoc}). The complex $(BA\otimes M)^{[\xi]}$ is acyclic and is the dual of the standard bar-resolution of $M^*$ as a right $A$--module.

Similarly we can form $\Hochb(A,A)=(\overline{B}A\otimes A)^\xi$ and its left dg module $(\overline{B}A\otimes M)^{[\xi]}$, which can be identified with the dual \emph{reduced} standard bar-resolution $M^*$ as a right $A$--module.
\end{enumerate}
\end{example}
\begin{prop}\label{obprop}
Let $A$ be a curved pseudo-compact algebra and $\xi\in A^1$. Then the category of $A$--modules and the category of $A^\xi$--modules are Quillen equivalent. To an $A$--module $N$ we associate an $A^\xi$--module $M^{[\xi]}$ and to an $A^\xi$--module $N$ we associate the module $N^{[-\xi]}$.
\end{prop}
\begin{proof}
Clearly the categories are isomorphic: the $A$--module $(M^{[\xi]})^{[-\xi]}$ is equal to $M$ and similarly with an $A^\xi$--module $N$.

As both the categories are of pseudo-compact modules, cofibrations are characterized as the injective morphisms with projective cokernels (disregarding differentials).  Clearly this property is preserved by both the functors described in the proposition.
\end{proof}

\section{Koszul duality}
Let $A$ be a unital dg algebra and $B:=\overline{B}A$ be its reduced bar-construction, a curved local pseudo-compact algebra.

\begin{defi}
We define a pair of (contravariant) functors between the categories of $A$--modules and of $B$--modules as follows.  The functor $F$ associates to an $A$--module $N$, the twisted module{\rm:} $$F(N)=(B\otimes N^*)^{[\xi]}.$$  Here we view $B\otimes N^*$ as a $B\otimes B^{\op}\otimes A^{\op}$--module and $\xi$ is the canonical element in $B^{\op}\otimes A^{\op}$. Thus, $(B\otimes N^*)^{[\xi]}$ becomes a $B\otimes(B^{\op}\otimes A^{\op})^{\xi}=B\otimes\Hochb(A,A)^{\op}$--module and forgetting the $\Hochb(A,A)^{\op}$--action we get a $B$--module.

The functor $G$ in the opposite direction associates to a $B$--module $L$ the twisted module{\rm:} $$G(L)=(A\otimes L^*)^{[\xi]}.$$ Here we view $A\otimes L^*$ as a $ B^{\op}\otimes A^{\op}\otimes A=\overline{B}A^{\op}\otimes A^{\op}\otimes A$--module; $\xi$ is the canonical element in $\overline{B}A^{\op}\otimes A^{\op}$. Thus, $(A\otimes L^*)^{[\xi]}$ is a $(\overline{B}A^{\op}\otimes A^{\op})^\xi\otimes A= \Hochb(A,A)^{\op}\otimes A$--module and forgetting the $\Hochb(A,A)^{\op}$--action we get an $A$--module.
\end{defi}

Explicitly then, if $\alpha\in F(N)$ is a map $A^{\times k} \times N \to \ground$, then $d\alpha$ maps:
\begin{eqnarray*}
(a_1,\cdots,a_{k+1},n) \mapsto &\,\,&\epsilon(a_1)\alpha (a_2,\cdots,a_{k+1},n) \\&-&\alpha(a_1a_2,\cdots,a_{k+1},n)\\&+&\cdots\\&+&(-1)^{k+1}\alpha(a_1,\cdots,a_{k+1},n),\end{eqnarray*}
where as before $\epsilon$ denotes a fake augmentation and the elements $a_i$, $i=1,\ldots, k+1$ are understood to be even, as before.  It is clear that $d^2$ induces application of the curvature element in $B$ rather than commutation with it, so $F(N)$ is indeed a $B$--module.

From \cite[Theorem 6.3a]{P} we know that $F,G$ form a mutually inverse pair of equivalences of categories.  This equivalence is then a Quillen equivalence with respect to the closed model category structure given by \cite[Theorem 9.3]{P}.  Thus we have:
\begin{theorem}
The functors $F,G$ form a (contravariant) Quillen equivalence between the categories of dg $A$--modules and dg $B$--modules.
\end{theorem}

\begin{rem}\begin{enumerate}
\item
The action of $\Hochb(A,A)$ was not mentioned in \cite{P} and the functors $F$ and $G$ were described in a different, although equivalent, language. The role of the additional right action of $\Hochb(A,A)$ is not completely clear.  One can speculate that since $\Hochb(A,A)$ is quasi-isomorphic to a strictly commutative dg algebra, its left and right actions coincide in some strong homotopy sense and moreover, are formal consequences (again in an appropriate strong homotopy sense) of the actions of $A$ or $B$.
\item
The action of the curved dg algebra $B$ on $F(N)$ is obtained as a restriction using the inclusion of curved dg algebras $B\hookrightarrow B\otimes\Hochb(A,A)\colon b\mapsto b\otimes 1$. Note that this required the twist: the corresponding map $B\to B\otimes B^{\op}\otimes A^{\op}$ is \emph{not} a map of curved dg algebras and $B\otimes N^*$ is \emph{not} a $B$--module, just as $B$ is not a $B$--module.
\end{enumerate}
\end{rem}

\section{Covariant Morita equivalence for modules over a pseudo-compact algebra}
Let $B$ be a pseudo-compact curved dg algebra and $M$ be a non-zero finite-dimensional vector space over $\ground$. Then $E':=B\otimes \End M$ is a curved pseudo-compact algebra. We have:
\begin{prop}\label{usual101}
The categories of $B$--modules and of $E'$--modules are Quillen equivalent. To a $B$--module $N$ we associate the $E'$--module $F^\prime(N):=N\otimes M$. To an $E'$--module $L$ we associate a $B$--module $G^\prime(L):=
\Hom_E(B\otimes M, L)\cong \Hom_{\End M}(M, L)$.
\end{prop}

\begin{proof}
We have natural isomorphisms:
\begin{eqnarray*}
G'F'(N)&=&\Hom_{\End M}(M, N\otimes M)\\
&\cong&\Hom_{\End M}(M, M)\otimes N\cong \ground \otimes N=N,\\\\
F'G'(L)&=&\Hom_{\End M}(M, L)\otimes M\\
&\cong&\Hom(M^*,\Hom_{\End M}(M,L))\\
&\cong&\Hom_{\End M}(M^*\otimes M,L)\\
&\cong&\Hom_{\End M}(\End M,L)=L.
\end{eqnarray*}

It remains to verify that $F'$ preserves cofibrations and that $G'$ preserves fibrations.  As a module over itself $\End M$ is a direct sum of copies of $M$.  So $M$ is a projective $\End M$ module.

Thus given a surjective map $f\colon L_1 \to L_2$, the induced map $G(f)$ is surjective, as we may lift any element $g \in \Hom_{\End M}(M, L_2)$:
$$
\xymatrix{
&M\ar[d]^g\ar@{.>}[dl]\\
L_1\ar[r]^f&L_2
}
$$
Thus $G'$ preserves fibrations.

A cofibration of $B$--modules $f\colon N_1 \to N_2$ is characterized by being part of an exact sequence:$$0 \to N_1 \stackrel{f}\to N_2 \to P \to 0,$$
for some $B$-module that is projective as a graded $B$-module, forgetting differentials.
  Applying $F'$ we obtain the exact sequence:
$$0 \to N_1 \otimes M\stackrel{f\otimes 1_M}\to N_2 \otimes M\to P\otimes M \to 0.$$
To deduce that $F'$ preserves cofibrations we need only note that $P \otimes M$ is a summand of the projective module $E'$--module $P \otimes \End M$.
\end{proof}

\section{Composing the equivalences}
We now have all the ingredients to prove our main theorem. We have the following chain of Quillen equivalences:$$
A\modcat \leftrightarrow (\overline{B}A\modcat)^{\op}\leftrightarrow (\overline{B}A\otimes\End M\modcat)^{\op}.$$

Note that both $\overline{B}A$ and $\overline{B}A\otimes\End M$ are curved pseudo-compact algebras. Using Example \ref{bigexample} (3) we recognize
$\Hochb(A,\End M)$ as an appropriate twisting $(\overline{B}A\otimes\End M)^\xi$ and using Proposition \ref{obprop} conclude that $E':=\overline{B}A\otimes\End M$--modules are Quillen equivalent to $E:=\Hochb(A,\End M)$--modules.

Moreover, the composition of functors from left to right is clearly $F$ as defined in Theorem \ref{mainres}.  This completes the proof of Theorem \ref{mainres}.

Tracing the functors going in the opposite direction we can give the following explicit description of the inverse equivalence $(E\modcat)^{\op}\to A\modcat$.  Let $L$ be in $E$-mod.  Considering the canonical element $\xi$ in $E$ we have $E^{-\xi}=\overline{B}A\otimes\End M$ and so $L^{[-\xi]}$ is a module over the curved dg algebra $\overline{B}A\otimes\End M$.

Applying $G'$ returns the $\overline{B}A$ module $\Hom_{\End M}(M,L^{[-\xi]})$.  In order to dualize this we employ the following lemma.

\begin{lem}
Given an $\End M$--module K we have{\rm:}
$\Hom_{\End M}(M,K)^* \cong \Hom_{\End M}(K,M)$.
\end{lem}

\begin{proof}
Replacing $B$ in Proposition \ref{usual101} with $\ground$ gives us the Morita duality between $\ground$ and $\End M$: $$\Hom_{\End M}(K,M)\cong \Hom_\ground(\Hom_{\End M}(M,K), \ground)\cong\Hom_{\End M}(M,K)^*.$$
\end{proof}

Thus applying $G$ to $\Hom_{\End M}(M,L^{[-\xi]})$, we get a twist of the $A\otimes A^{\op}\otimes \overline{B}A^{\op}$--module:
$$A\otimes\Hom_{\End M}(L^{[-\xi]},M).$$

Applying this (un)twist, we obtain that $(A\otimes\Hom_{\End M}(L^{[-\xi]},M))^{[\eta]}$ is a module over  $A\otimes (A^{\op}\otimes \overline{B}A^{\op})^{\eta}=A\otimes\Hochb(A,A)^{\op}$ (where $\eta$ is the corresponding MC element).  Finally, forgetting the action of $\Hochb(A,A)$, we obtain an $A$--module $(A\otimes\Hom_{\End M}(L^{[-\xi]},M))^{[\eta]}$. This is the value of our composite functor $(E\modcat)^{\op}\to A\modcat$ on $L$.

\section{Equivalences between categories of modules and pseudo-compact modules}
We saw that the category of dg modules over a dg algebra is always Quillen equivalent to some category of pseudo-compact modules (Corollary \ref{modcomod}). It is natural to ask when, conversely, the category of pseudo-compact modules over a pseudo-compact dg algebra is Quillen equivalent to the category of dg modules over some dg algebra. The following result gives an answer to that question; here and later on $D(A)$ and $D(B)$ stand for the derived categories of dg $A$-modules and of pseudo-compact $B$-modules respectively, i.e.\ the homotopy categories of the corresponding closed model structures.

 \begin{theorem}\label{BtoA}
Let $B$ be a pseudo-compact dg algebra.
The following are equivalent{\rm:}
\begin{enumerate}
\item
There exists a dg algebra $A$ and an equivalence
$F\colon D(A)\cong D(B)^{\op}$,
\item
$D(B)^{\op}$ admits a compact generator,
\item
$D(B)^{\op}$ admits a finite-dimensional (necessarily compact) generator.
\end{enumerate}
If any of these conditions holds then in fact there exists a dg algebra $A$ and a Quillen equivalence between $A$-mod and $(B\modcat)^{\op}.$

In case $B=B^0$, i.e.\ $B$ is an ordinary pseudo-compact algebra,  any of the three conditions above is equivalent to the following statement{\rm:}
\begin{enumerate}
\item[(4)]
There are finitely many isomorphism classes of simple (non-dg) $B$--modules.
\end{enumerate}
\end{theorem}

\begin{lem}\label{cogeneration}
Let $B$ be pseudo-compact dg algebra.
Then $D(B)^{\op}$ is compactly generated, and an object of $D(B)^{\op}$ is compact if and only if it is in the thick subcategory of $D(B)^{\op}$ generated by (totally) finite-dimensional modules.
\end{lem}
\begin{proof}
By {\cite[\S5.5]{P}}, the finite-dimensional $B$--modules form a set of compact objects that generate $D(B)^{\op}$.
It follows that the compact objects in $D(B)^{\op}$ are generated as a thick subcategory of $D(B)^{\op}$ by the finite-dimensional modules \cite[Theorem 2.1.3(c)]{HPS}.
\end{proof}

\begin{proof}[Proof of Theorem \ref{BtoA}] Given an equivalence $F$, we have that $F(A)$ is a compact generator of $D(B)^{\op}$. Hence, by Lemma~\ref{cogeneration}, $F(A)$ is obtained from a finite set of totally finite-dimensional $B$--modules by a sequence of shifts, extensions and retractions. Thus the direct sum of these finite-dimensional modules is a compact generator of $D(B)^{\op}$.
This proves that (1) implies (3), and (3) obviously implies (2).

Now suppose (2) holds: let $N$ be a compact generator of $D(B)^{\op}$. Without loss of generality we may assume $N$ is cofibrant. Define a dg algebra $A:=\End_B(N)$. Then the functor
$$F(-)=\Hom_A(-,N)\colon A\modcat\rightarrow (B\modcat)^{\op}$$
is left adjoint to the functor
$$G(-)=\Hom_B(-,N)\colon (B\modcat)^{\op}\rightarrow A\modcat.$$
Recall \cite{P} that cofibrations in $A$-mod are the injective maps with cofibrant (in particular projective) cokernel, whereas fibrations are the surjective maps.
Similarly
fibrations in $(B\modcat)^{\op}$ are the injective maps with projective cokernel, and fibrations are the surjective maps. It is then easy to see that
$F$ preserves cofibrations and $G$ preserves fibrations. Moreover $F(A)\cong N$ and $G(N)\cong A$. Since $A$ and $N$ are compact generators in $D(A)$ and  $D(B)$ respectively, it follows that $F$ and $G$ induce inverse equivalences of $D(A)$ and $D(B)^{\op}$.

Suppose now that $B$ is an ordinary pseudo-compact algebra. Simple $B$--modules may be regarded as simple dg $B$--modules concentrated in degree $0$, and any simple dg $B$--module arises in this way, up to grading shift.
If $B$ has finitely many isomorphism classes of simple modules, their direct sum is then a compact generator, as any finite-dimensional dg $B$--module has a finite composition series.

Conversely suppose $B$ has a finite-dimensional (compact) generator $N$.
It suffices to show that any simple $B$--module $S$ is a composition factor of the cohomology $H(N)$, regarded as a (finite-dimensional) ungraded $B$--module. To confirm the latter, note that $S$, as a compact object in $D(B)^{\op}$, is contained in the thick subcategory of $D(B)^{\op}$ generated by $N$, and observe that shifts, retracts and extensions of dg $B$--modules cannot create new composition factors in cohomology.
\end{proof}

\begin{rem} The obvious generalisation  of condition (4) in Proposition~\ref{BtoA}
to arbitrary pseudcompact dg algebras $B$
would be the following{\rm:}
\begin{enumerate}
\item[(4')]
There are finitely many isomorphism classes of simple dg $B$--modules, up to grading shift.
\end{enumerate}
As a case in point, the pseudo-compact dg algebra $E=\Hochb^\bullet(A,\End M)$ of Theorem~\ref{mainres} has a unique simple dg module up to isomorphism and grading shift -- the module $M$.
\end{rem}

\begin{rem}
Let $G$ be an affine group over a field $k$. The category of $G$--modules may be identified with the category of $k[G]$-comodules, or equivalently, with the category of modules over the pseudo-compact algebra $k[G]^*$. Thus it is natural to consider the derived category $D(k[G]^*)$ and to ask for which $G$ it admits a compact generator. By Theorem~\ref{BtoA} this holds if and only if there are finitely many isomorphism classes of simple $G$--modules.

In case $G$ is a smooth affine algebraic group over an algebraically closed field, we can give the following answer: $G$ has finitely many simple modules if and only if $G^0$, the connected component of the identity, is unipotent. Indeed, since $G^0$ is a normal subgroup of finite index in $G$, the standard arguments of Clifford theory imply that $G$ has finitely many simple modules precisely when $G^0$ does. If $G^0$ is unipotent it has a unique simple module. On the other hand if $G^0$ is not unipotent, it has a nontrivial reductive quotient and thus infinitely many isomorphism classes of simple modules.
\end{rem}
\section{Comparison with classical Morita duality} Our main result, Theorem \ref{mainres}, is a Quillen anti-equivalence, or duality between two module categories; it is given as a kind of derived Hom functor. This suggests a close relationship with Morita duality
 \cite{Morita, Lam} which also studies contravariant equivalences between various categories of modules. We will see that that our result can indeed be viewed as an extension of Morita duality in the case when the algebra in question is finite dimensional and of finite global dimension.

 Let us first present a kind of derived Morita duality when $A$ is an ordinary (i.e.\ non-dg) finite-dimensional algebra; this will be our standing assumption in this section. If $M$ is a finite-dimensional injective cogenerator of the category of $A$ modules and $\Gamma:=\End_A(M)$ then the category of finite-dimensional $A$--modules is anti-equivalent to the category of finite-dimensional $\Gamma$--modules via the functors
$F\colon N\mapsto \Hom_A(N,M)$ and $G\colon L\mapsto \Hom_{\Gamma}(L,M)$ (see e.g.\ \cite[Theorem 7.11]{Sk}). Note that since $\Gamma$ is a finite dimensional algebra, it makes sense to consider the category of its (left) pseudo-compact modules which we will denote by $\Gamma^{\ps}$--mod; note that its opposite category is naturally identified with the category of  $\Gamma$-comodules.   The following result is an easy extension of this version of Morita duality.
\begin{theorem}
The functors $F$ and $G$ determine an anti-equivalence between the abelian categories $A$--mod and $\Gamma^{\ps}$--mod.
\end{theorem}
\begin{proof}
The functor
$G\colon \Gamma^{\ps}$--mod$\to A$--mod can be factored as a composition, \[\Gamma^{\ps}\modcat\longrightarrow
\Gamma\modcat\longrightarrow A\modcat,\] as follows:
\[
L\mapsto L^*\mapsto \Hom_{\Gamma}(M^*,L^*)\cong\Hom_{\Gamma}(L,M),
\]
where $L$ is a pseudo-compact $\Gamma$--module. The functor of linear duality $L\to L^*$ is clearly an anti-equivalence between $\Gamma^{\ps}$--mod and $\Gamma$--mod whereas the functor $\Hom_{\Gamma}(M^*,-)$ is the usual covariant Morita equivalence between $\Gamma$--mod and $A$--mod (note that $M^*$ is a projective generator of $A$--mod since $M$ is an injective cogenerator).
\end{proof}
\begin{rem}\label{absequiv}
It follows that $F$ and $G$ determine an anti-equivalence between the homotopy categories of complexes in $A$-mod and $\Gamma^{\ps}$-mod. Taking the Verdier quotient by the acyclic complexes, we conclude that $D(A)$, the derived category of $A$, is anti-equivalent to $D^I(\Gamma^{\ps})$, the derived category of pseudo-compact $\Gamma$--modules of the \emph{first kind} (cf. \cite{P} concerning this terminology). Note that our previous results were concerned with the derived categories of pseudo-compact modules of the \emph{second kind}.

Recall from \cite{P} that a complex of $A$-modules is \emph{absolutely acyclic} if it belongs
to the minimal thick subcategory of the category $\operatorname{Hot}(A\modcat)$ of $A$-modules up to homotopy,
 containing acyclic bounded complexes. The Verdier quotient of $\operatorname{Hot}(A\modcat)$ by the subcategory of absolutely acyclic complexes is called the \emph{absolute derived category} of $A$ and denoted by $D^{\abs}(A)$. In the same way we can define the absolute derived category $D^{\abs}(\Gamma^{\ps})$ of pseudo-compact $\Gamma$-modules. It follows similarly that the functors $F$ and $G$ determine an anti-equivalence between  $D^{\abs}(A)$ and $D^{\abs}(\Gamma^{\ps})$.
\end{rem}
We observe that there is a simple criterion for the two types of derived categories to coincide.
\begin{lem}\label{absacyclic}
If a finite-dimensional algebra $\Gamma$ has finite global dimension then the following categories coincide{\rm:}
\begin{enumerate}
\item
$D^{\abs}(\Gamma)$ and $D(\Gamma)${\rm;}
\item
$D^{\abs}(\Gamma^{\ps})$, $D(\Gamma^{\ps})$ and $D^I(\Gamma^{\ps})$.
\end{enumerate}
\end{lem}
\begin{proof}
Note first of all that  (1)$\implies$(2) by Remark \ref{absequiv}.  Thus, it suffices to prove (1). To this end let $M$ be an acyclic complex of  $\Gamma$--modules; we have to show that $\Gamma$ is absolutely acyclic. Let $M\langle n,m\rangle$ be the complex of $\Gamma$--modules such that{\rm :}
\[
M\langle n,m\rangle^i=\begin{cases} M^i \text{~if~} n<i<m \\ \ker d\colon M^n\to M^{n+1} \text{~if~}i=n\\ \operatorname{Coker} d\colon M^{m-1}\to M^m \text{~if~}i=m\\0\text{~if~}i<n \text{~or~} i>m\end{cases}
\]
Then clearly $M\langle n,m\rangle$ is absolutely acyclic and
\[
M\simeq\operatorname{holim}_{ n}\operatorname{hocolim}_{m}M\langle n,m\rangle.
\]
It follows by \cite[Theorem 3.6]{P} that the category of absolutely acyclic complexes is closed with respect to arbitrary direct sums and direct products, and therefore also with respect to homotopy limits and colimits along directed systems . Therefore $M$ is absolutely acyclic.
\end{proof}
\begin{rem}
If $\Gamma$ is a pseudo-compact dg algebra of finite global dimension then Positselski \cite[Theorem 3.6 and 4.5]{P} showed that various derived categories of second kind of $\Gamma$ coincide. The result above shows that under the additional assumption that $\Gamma$ is concentrated in degree zero, these derived categories of second kind also coincide with the ordinary derived categories.
\end{rem}
We can now formulate the main result of this section.
\begin{theorem}\label{finitedim}\
\begin{enumerate}
\item
Let $A$ be a finite-dimensional algebra of finite global dimension. Then the categories of dg $A$--modules and of pseudo-compact dg $A$--modules are Quillen anti-equivalent.
 \item Conversely, suppose that there exists a finite-dimensional algebra $\Gamma$ such that  the categories $D(A)$ and $D(\Gamma^{\ps})^{\op}$ are
 equivalent. Then $A$ and $\Gamma$ both have finite global dimension, and $D(A)$ and $D(\Gamma)$ are equivalent.
 \end{enumerate}
\end{theorem}
\begin{proof}
Note that the $\ground$-linear duality functor $L\mapsto L^*$ determines an anti-equivalence between the abelian categories $A\modcat$ and $A^{\ps}\modcat$; it is also a Quillen functor between the corresponding closed model categories. Now the homotopy category of $A\modcat$ (in the sense of closed model categories) is the derived category $D(A)$. By Lemma \ref{absacyclic} the homotopy category of the closed model category $A^{\ps}\modcat$ is the same as $D^I(A^{\ps})$ which is then equivalent to $D(A)$ under the functor of linear duality. This proves (1).

Now suppose that $D(A)$ and $D(\Gamma^{\ps})$ are equivalent. Then arbitrary products and coproducts exist in $D(\Gamma^{\ps})$.
Since absolutely acyclic pseudo-compact $\Gamma$--modules vanish in $D(\Gamma^{\ps})$, arguing as in the proof of Lemma~\ref{absacyclic} we see, that the same is true of all acyclic pseudo-compact $\Gamma$--modules. Thus $D(\Gamma^{\ps})$ coincides with $D^I(\Gamma^{\ps})$, and is thus anti-equivalent to $D(\Gamma)$ via linear duality. Now recall from Lemma~\ref{cogeneration} that finite-dimensional $\Gamma$-modules are compact objects in $D(\Gamma^{\ps})^{\op}$. We deduce from the equivalence $D(\Gamma)\cong D(\Gamma^{\ps})^{\op}\colon M\mapsto M^*$ that any finite-dimensional dg $A$--module is compact in $D(A)$, which implies that $A$ has finite global dimension, and therefore so does $\Gamma$.
\end{proof}
\begin{rem}
For a finite-dimensional algebra $A$ there are two natural choices for a finite-dimensional module $M$; namely one can take $M=A$ or $M=A^*$; the latter choice having the advantage of being an injective cogenerator. Then Theorem \ref{mainres} states that the category $A$--mod and $E$--mod where
 $E:=\overline{\Hoch}(A,\End(M))$ are Quillen equivalent. Since $E$ is quasi-isomorphic to $\operatorname{RHom}_A(M,M)$, for
 $M=A$ or $M=A^*$ it is further quasi-isomorphic to $A$.  One can ask whether the above equivalence
simplifies to a Quillen equivalence between $A$--modules and pseudo-compact $A$--modules. Theorem \ref{finitedim} says, in particular, that this is the case only when $A$ has finite global dimension.
\end{rem}

\end{document}